\documentclass[a4paper,12pt]{amsart}
\usepackage{geometry}

\usepackage{amsmath,amsthm,amsfonts,amssymb,amscd}
\usepackage{fullpage}
\usepackage{lastpage}
\usepackage{enumerate}
\usepackage{fancyhdr}
      
\usepackage[american]{babel}

\usepackage{mathrsfs}
\usepackage{epsf}
\usepackage{eso-pic,graphicx,color}
\usepackage{pstricks}
\usepackage{multirow}
\usepackage{latexsym,epsfig,hyperref}          

\usepackage{amsopn}
\usepackage{textcomp}

\usepackage{palatino}

\newtheorem{thm}{Theorem}
\newtheorem{remark}[thm]{Remark}
\newtheorem{prop}[thm]{Proposition}

\newtheorem{coro}[thm]{Corollary}

\newtheorem{lemma}[thm]{Lemma}

\newtheorem*{theoremA}{Theorem A}

\def\dim{\operatorname{dim}}

\def\corank{\operatorname{corank}}

\def\Im{\operatorname{Im}}

\def\Pic{\operatorname{Pic}}%
\def\Aut{\operatorname{Aut}}%
\def\codim{\operatorname{codim}}%

\newcommand \lra {\rightarrow}
\title{Components of the Hilbert Scheme of smooth projective curves using ruled surfaces}

\author[Y. Choi]{Youngook Choi}
\address{Department of Mathematics Education, Yeungnam University, 280 Daehak-Ro, \hfill \newline\texttt{}
 \indent Gyeongsan, Gyeongbuk 38541, Republic of Korea}
\email{ychoi824@yu.ac.kr}

\author[H. Iliev]{Hristo Iliev}
\address{American University in Bulgaria, 2700 Blagoevgrad, Bulgaria, and \hfill \newline\texttt{}
 \indent Institute of Mathematics and Informatics, Bulgarian Academy of Sciences, \hfill \newline\texttt{}
 \indent 1113 Sofia, Bulgaria}
\email{ hiliev@aubg.edu, hki@math.bas.bg}

\author[S. Kim]{Seonja Kim}
\address{Department of Electronic Engineering, Chungwoon University, Sukgol-ro, Nam-gu, \hfill \newline\texttt{}
 \indent Incheon 22100, Republic of Korea}
\email{sjkim@chungwoon.ac.kr}

\thanks{The first author was supported by Basic Science Research Program through the National Research Foundation of Korea(NRF) 
funded by the Ministry of Education(NRF-2016R1D1A3B03933342).  The third  author was supported by Basic Science Research 
Program through the National Research Foundation of Korea(NRF) funded by the Ministry of Education (NRF-2016R1D1A1B03930844).}

\subjclass[2000]{Primary 14C05; Secondary 14H10}
\keywords{Hilbert scheme of curves, double covering, ruled surfaces}

\usepackage{setspace}
\begin{document}

\setlength{\parindent}{5ex}

\begin{abstract}
Let $\mathcal{I}_{d,g,r}$ be the union of irreducible components of the Hilbert scheme whose general points correspond to smooth 
irreducible non-degenerate curves of degree $d$ and genus $g$ in $\mathbb{P}^r$. We use families of curves on cones to show that 
under certain numerical assumptions for $d$, $g$ and $r$, the scheme $\mathcal{I}_{d,g,r}$ acquires generically smooth 
components whose general points correspond to curves that are double covers of irrational curves. In particular, in the case 
$\rho(d,g,r) := g-(r+1)(g-d+r) \geq 0$ we construct explicitly a \emph{regular component} that is different from the 
distinguished component of $\mathcal{I}_{d,g,r}$ dominating the moduli space $\mathcal{M}_g$. Our result implies also 
that if $g \geq 57$ then $\mathcal{I}_{\frac{4g}{3}, g, \frac{g+1}{2}}$ has at least two generically smooth 
components parametrizing linearly normal curves.
\end{abstract}

\subjclass[2000]{Primary 14C05; Secondary 14H10}
\keywords{Hilbert scheme of curves, Brill-Noether theory, double covering}

\maketitle

\section{Introduction}\label{Section_1}

Let $\mathcal{I}_{d,g,r}$ be the union of irreducible components of  the Hilbert scheme whose general points correspond 
to smooth irreducible non-degenerate complex curves of degree $d$ and genus $g$  in $\mathbb{P}^r$. A component of 
$\mathcal{I}_{d,g,r}$ is called \emph{regular} if it is reduced and of expected dimension $\lambda_{d,g,r} := (r+1)d - 
(r-3)(g-1)$. 
Otherwise it is called \emph{superabundant}.
For $\rho(d,g,r) := g-(r+1)(g-d+r) \ge 0$, it is known that $\mathcal{I}_{d,g,r}$ has the unique component dominating
$\mathcal{M}_g$, see \cite[p. 70]{Har82}. It is usually referred to as the \emph{distinguished} component.

Historically, Severi claimed in \cite{Sev1921} that $\mathcal{I}_{d,g,r}$ is irreducible if $d\geq g+r$. 
It was proved that $\mathcal{I}_{d,g,r}$ is irreducible if $d\geq g+r$ and $r = 3, 4$, see \cite{Ein86} and \cite{Ein87}.
On the other hand,  for $r\geq 5$ (and $\rho(d,g,r)\geq 0$) there have been given several examples in which $\mathcal{I}_{d,g,r}$ possesses additional 
non-distinguished components (\cite{MS1989}, \cite{Kee94},  \cite{CIK17}, \cite{CS89},  etc), but for none of them it has been proven to be 
regular. Note that all these examples are given
by non-linearly normal curves.  We remark that in \cite{CIK17} we showed the existence of a {\it non-distinguished} component $\mathcal{D}_{d,g,r}$ of 
$\mathcal{I}_{d,g,r}$ parameterizing curves that are double covers of irrational curves, whereas all other known to us examples 
of reducible Hilbert schemes of curves have used curves that are $m$-sheeted coverings of $\mathbb{P}^1$ with $m \geq 3$.

In \cite[Question 4.7, p. 598]{CIK17} we asked about the possibility of $\mathcal{D}_{d,g,r}$ being reduced. In the present 
paper we reconstruct this component under less constrains, this time using a family of curves on cones, which are double 
coverings of hyperplane sections of the cones. We construct and characterize the properties of $\mathcal{D}_{d,g,r}$ using tools 
from the theory of ruled surfaces, while in \cite{CIK17} we only showed its existence using Brill-Noether theory of linear 
series on curves. Our approach is motivated by the fact that for a given double covering $\varphi : X \to Y$ the curves $X$ and 
$Y$ can be regarded as curves on the ruled surface $S := \mathbb{P} (\varphi_{\ast} {\mathcal O _X})$, as we explain in section 
\ref{Section_2}. It allows us to construct the additional component in a more geometric way and to obtain its generic 
smoothness, which gives  an affirmative answer to the  question raised in \cite{CIK17}.

Our main result is as follows.

\begin{theoremA} 
Assume that $g$ and $\gamma$ are integers with $g \geq 4\gamma -2 \geq 38$.
Let
\[
  d:= 2g - 4\gamma + 2 \quad \mbox{ and } \quad \max \left\lbrace \gamma, \frac{2(g-1)}{\gamma} \right\rbrace \leq r \leq R:= g 
- 3\gamma + 2 \, .
\] 
Then the Hilbert scheme $\mathcal{I}_{d, g, r}$ possesses a generically reduced component $\mathcal{D}_{d, g, r}$ for which
\[ 
  \dim \mathcal{D}_{d, g, r} = \lambda_{d, g, r} + r\gamma - 2g+2 \, .
\] 
Further, let $X_r \subset \mathbb{P}^r$ be a smooth curve corresponding to a general point  of $\mathcal{D}_{d, g, r}$.
\begin{enumerate}
  \item[{\rm (i)}] If $r = R$ then $X_R$ is the intersection of a general quadric  hypersurface with a cone over a smooth 
curve $Y$ of degree $g - 2\gamma + 1$ and genus $\gamma$ in $\mathbb{P}^{R-1}$ and $X_R$ is embedded in $\mathbb{P}^R$ by the 
complete linear series $|R_{\varphi}|$ on $X_R$, where $R_{\varphi}$ is the ramification divisor of the natural projection 
morphism $\varphi : X_R \to Y$ of degree 2 given by the ruling of cone;
  \item[{\rm (ii)}] If $r < R$ then $X_r$ is given by a general projection of some $X_R$ as in {\rm (i)}, that is, $X_r$ is 
embedded in $\mathbb{P}^{r}$ by a general linear subseries $g^r_d$ of $|R_{\varphi}|$.
\end{enumerate} 
\end{theoremA}

In our view, one of the interesting implications of {\rm Theorem A} is that if $r = \frac{2(g-1)}{\gamma} \geq \gamma 
\geq 10$ and $d = 2g - 4\gamma + 2$, then the scheme $\mathcal{I}_{d, g, r}$ acquires a second regular component 
in addition to its distinguished component dominating the moduli space $\mathcal{M}_g$, see {\rm Corollary 
\ref{Coro_Reg_comp_exists}}. To our best knowledge, it is the first example in which simultaneous existence of two distinct 
regular components of $\mathcal{I}_{d,g,r}$ has been observed in the Brill-Noether case $\rho(d,g,r) \geq 0$. We remark also 
that in the case $g = 6\gamma - 3$ and $r = R = 3\gamma - 1$, the Hilbert scheme $\mathcal{I}_{\frac{4g}{3}, g, \frac{g+1}{2}}$ 
has at least two generically smooth components parametrizing linearly normal curves as it is explained in {Remark 
\ref{Sec3_2_LinNorm_comp_exists}}.

The remaining sections of the paper are organized as follows. In section \ref{Section_2}, we provide a motivation for the 
construction of the component described in {\rm Theorem A} by reviewing the relations between double coverings of curves, ruled 
surfaces and their embeddings as cones. We also prove there several statements that will be used for the construction of
$\mathcal{D}_{d, g, r}$ in section \ref{Section_4}. Possibly, some of them might be of independent interest. In section 
\ref{Section_3} we briefly review several facts about the Gaussian map associated to linear series on curves and prove a 
technical result facilitating the computation of the dimension of the tangent space at a general point of $\mathcal{D}_{d, g, 
r}$. In section \ref{Section_4} we give the proof of {\rm Theorem A}.

We work over $\mathbb{C}$. We understand by \emph{curve} a smooth integral projective algebraic curve. We denote by $L^{\vee}$ 
the dual line bundle for a given line bundle $L$ defined on an algebraic variety $X$. As usual, $\omega_X$ will stand for the 
canonical line bundle on $X$. We denote by $|L|$ the complete linear series $\mathbb P\left(H^0(X,L)\right)$. When $X$ is an 
object of a family, we denote by $[X]$ the corresponding point of the Hilbert scheme representing the family. Throughout the 
entire paper
\[
  d := 2g - 4\gamma + 2 \quad \mbox{ and } \quad R := g - 3\gamma + 2 \, .
\]
For definitions and properties of the objects not explicitly introduced in the paper refer to \cite{Hart77} and \cite{ACGH}.

\noindent
{\bf Acknowledgements}

We thank KIAS for the warm hospitality when we were associate members in KIAS and the second author visited 
there. We would like to thank the referees for the constructive comments and valuable suggestions, which helped to 
improve the quality of our paper.

\bigskip

\section{Motivation and preliminary results}\label{Section_2}

Suppose that $\varphi : X \to Y$ is an $m:1$ cover, $m \geq 2$, where $X$ and $Y$ are smooth curves of genus $g$ and $\gamma$, 
correspondingly. As it is well known, the covering induces a short exact sequence of vector bundles on $Y$
\[
 0 \to \mathcal{O}_Y \xrightarrow{\varphi^{\sharp}} \varphi_{\ast} \mathcal{O}_X \to \mathcal{E}^{\vee} \to 0 \, ,
\]
where $\mathcal{E}^{\vee}$ is the so called \emph{Tschirnhausen module},  see \cite{Mir85}. It is a rank $(m-1)$-vector bundle on $Y$. 
Since $X$ and $Y$ are curves over $\mathbb{C}$, the exact sequence splits, i.e. $\varphi_{\ast} \mathcal{O}_X \cong 
\mathcal{O}_Y \oplus \mathcal{E}^{\vee}$. According to \cite[Ex. IV.2.6, p. 306]{Hart77}, $(\det \varphi_{\ast} 
\mathcal{O}_X )^2 \cong \mathcal{O}_Y (-B)$, where $B$ is the branch divisor of the covering. In 
particular, $\deg B = 2(g - 1) - 2m (\gamma - 1)$.

We focus on the case $m = 2$. In such a case $\mathcal{E}$ must be a line bundle on $Y$ and we can assume that 
$\mathcal{E} = \mathcal{O}_Y (E)$ for some divisor $E$ on $Y$. Since 
\[
 \deg B = - \deg (\det \varphi_{\ast} \mathcal{O}_X )^2
	= - \deg (\det ( \mathcal{O}_Y \oplus \mathcal{O}_Y (-E) ) )^2 
	= 2 \deg E
\]
it follows that $\deg E = g - 2\gamma + 1$.

Further we suppose that $E$ is a nonspecial and very ample divisor on $Y$. Denote by $\mathcal{F}$ the rank 2 vector bundle 
$\mathcal{F} := \mathcal{O}_Y \oplus \mathcal{O}_Y (E)$ on $Y$ and let $S$ be the ruled surface $S := \mathbb{P}(\mathcal{F})$ 
with natural projection $f : S \to Y$. Since $\deg E > 0$,
$\mathcal{F}_0 := \mathcal{O}_Y \oplus \mathcal{O}_Y (-E)$ will be the normalization of the vector bundle $\mathcal{F}$. As it 
is decomposable, $f \, : \, S \to Y$ has two canonically determined sections. They are $Y_0$ which corresponds to the 
short exact sequence
\[
    0 \to \mathcal{O}_Y \to \mathcal{O}_Y \oplus \mathcal{O}_Y (-E) \to \mathcal{O}_Y (-E) \to 0 \, ,
\]
and $Y_1$ which corresponds to the short exact sequence
\[
    0 \to \mathcal{O}_Y \to \mathcal{O}_Y \oplus \mathcal{O}_Y (E) \to \mathcal{O}_Y (E) \to 0 \, .
\]
The section $Y_0$ is the section with minimal self-intersection on $S$ and $Y^2_0 = \deg (\mathcal{O}_Y (-E)) = -g + 
2\gamma - 1$. As it well known, $\Pic (S) \cong \mathbb{Z}[Y_0] \oplus f^{\ast} (\Pic (Y))$. For a divisor $D$ on $Y$ 
we will denote by $D \mathfrak{f}$ the divisor $f^{\ast} (D)$ on $S$. Also, we have for the section $Y_1$ that $Y^2_1 = \deg 
(\mathcal{O}_Y (E)) = g - 2\gamma + 1$ and it is not difficult to see that $Y_1 \sim Y_0 + E \mathfrak{f}$. In general, 
cohomologies 
like $h^{j} (S, \mathcal{O}_S (nY_0 + D \mathfrak{f}))$ are calculated using the projection formula, see \cite[Ex. III.8.3, p. 
253]{Hart77}, as
\[
  h^{j} (S, \mathcal{O}_S (nY_0 + D \mathfrak{f})) = h^{j} (Y, \mathcal{S}ym^{n} (\mathcal{F}_0) \otimes \mathcal{O}_Y (D)) \, ,
\]
but since $S$ is decomposable, i.e. $\mathcal{F}_0$ splits, the calculation reduces simply to
\begin{equation}\label{S_decompose_cohomologies}
 h^{j} (S, \mathcal{O}_S (nY_0 + D \mathfrak{f})) = \sum^{n}_{k=0} h^{j} (Y, \mathcal{O}_Y (D - kE)) \, ,
\end{equation}
see for example \cite{FGP05}. From here
\begin{equation}\label{S_cohomologies_O_S(Y_1)}
 h^{0} (S, \mathcal{O}_S (Y_1)) 
      = h^{0} (S, \mathcal{O}_S (Y_0 + E \mathfrak{f})) 
      = h^{0} (Y, \mathcal{O}_Y (E)) + h^{0} (Y, \mathcal{O}_Y) 
      = g - 3\gamma + 3\, .
\end{equation}
Using \cite[Ex. V.2.11 (a), p. 385]{Hart77}, we obtain that the linear series $|\mathcal{O}_S (Y_1))| \equiv 
|\mathcal{O}_S (Y_0 + E \mathfrak{f}))|$ is base point free. Therefore it defines a morphism 
\[
  \Psi := \Psi_{|\mathcal{O}_S (Y_1))|} \, : \, S \to \mathbb{P}^R \, ,
\] 
where $R = g - 3\gamma + 2$. Since $E$ is very ample, it follows by \cite[Proposition 23, p. 38]{FGP05} that $\Psi$ is 
isomorphism away from $Y_0$. Due to $Y_0 \cdot Y_1 = Y_0 \cdot (Y_0 + E \mathfrak{f}) = 0$, the morphism $\Psi$ contracts the 
curve $Y_0$ to 
a point. Therefore $F := \Psi(S) \subset \mathbb{P}^R$ is a cone of degree 
\[
 \deg F = Y_1 \cdot Y_1 = (Y_0 + E \mathfrak{f}) \cdot (Y_0 + E \mathfrak{f}) = \deg E = g - 2\gamma + 1
\]
over the image of a smooth integral curve from the linear series $|\mathcal{O}_S (Y_0 + E \mathfrak{f}))|$. 

By Bertini's theorem, $\Psi$ maps a general element of $|\mathcal{O}_S (Y_1))|$ to a smooth integral curve of genus 
$\gamma$, degree $g - 2\gamma + 1$, which is further linearly normally embedded in some hyperplane $\mathbb{P}^{R-1}$ of 
$\mathbb{P}^{R}$ due to (\ref{S_cohomologies_O_S(Y_1)}). A similar fact is true about 
a general element of $|\mathcal{O}_S (2Y_1))|$. Namely, a general $C \in |\mathcal{O}_S (2Y_1))| \equiv |\mathcal{O}_S 
(2Y_0 + 2E \mathfrak{f}))|$ is mapped by $\Psi$ to a smooth integral curve $\Psi (C)$ of genus $g$, degree $2g - 4\gamma + 2 = 
d$, which is
linearly normal in  $\mathbb{P}^{R}$. Indeed, since $Y_0 \cdot Y_1 = 0$ and $\Psi$ is isomorphism away from $Y_0$, it follows by 
Bertini's theorem that $\Psi (C)$ is smooth and integral. Its degree is $\deg \Psi (C) = 2 Y_1 \cdot Y_1 = 2g - 4\gamma + 2$, 
while by the adjunction formula
\[
 \deg C \cdot (K_S + C) = (2Y_1) \cdot (K_S + 2Y_1)
			= 2(2\gamma - 2) + 2g - 4\gamma + 2
			= 2g - 2
\]
we get that its genus is $g$. Finally, to see that $\Psi (C) \subset \mathbb{P}^{R}$ is linearly normal, consider the exact 
sequence
\[
 0 \to \mathcal{O}_S (-Y_0 - E \mathfrak{f}) \to \mathcal{O}_S (Y_0 + E \mathfrak{f}) \to \mathcal{O}_C (Y_0 + E 
\mathfrak{f}) \to 0 \, .
\]
It is sufficient to see that $h^1 (S, \mathcal{O}_S (-Y_0 - E \mathfrak{f})) = 0$, which is not difficult to obtain using the 
Serre duality.

The arguments above motivate the following statement.

\begin{prop}\label{Prop_ruled_surface} 
Assume that  $Y$ is a smooth curve of genus $\gamma$ and $E$ is a very ample non-special divisor on $Y$ of degree $e$. Let $S 
:= \mathbb{P}(\mathcal{O}_Y \oplus \mathcal{O}_Y (-E))$, $Y_0$ be the section of minimal self-intersection of the 
natural projection $f: S \to Y$ and $Y_1 \in |\mathcal{O}_S (Y_0 + E \mathfrak{f})|$ be a smooth integral 
curve. Let $\Psi := \Psi_{|\mathcal{O}_S (Y_0 + E \mathfrak{f})|}$ be the morphism induced by the complete linear series 
$|\mathcal{O}_S (Y_0 + E \mathfrak{f})|$. Then:
\begin{enumerate}
  \item[{\rm (a)}]  $|\mathcal{O}_S (Y_0 + E \mathfrak{f})|$ is base point free and of dimension $e - \gamma + 1$;
  \item[{\rm (b)}]  $\Psi$ is an isomorphism away from $Y_0$ and contracts $Y_0$ to a point in $\mathbb{P}^{e - \gamma +1} $, 
in particular, $\Psi (S)$ is a cone over $\Psi (Y_1)$;
  \item[{\rm (c)}]  for a general $C \in |\mathcal{O}_S (2Y_0 + 2E \mathfrak{f})|$
    \begin{itemize}
     \item[{\rm (c.1)}] $\Psi (C)$ is a linearly normal smooth irreducible curve of genus $2\gamma + e - 1$ and degree $2e$ in 
$\mathbb{P}^{e - \gamma + 1}$;
     \item[{\rm (c.2)}] the linear series $|\mathcal{O}_C (R_{\varphi})|$ on $C$ is traced by the linear series $|\mathcal{O}_S 
(Y_0 + E \mathfrak{f})|$ on $S$, 
where  $R_{\varphi}$ is the ramification divisor of the morphism $\varphi \, : \, C \to Y$ induced by the ruling of 
$S$.
    \end{itemize}
\end{enumerate}
\end{prop}
\begin{proof}
Statements {\rm (a)}, {\rm (b)} and {\rm (c.1)} are obtained by very similar arguments like those in the discussion preceding 
the proposition.  We only need to check {\rm (c.2)}.
Recall that $K_S \sim -2Y_0 + (K_Y - E) \mathfrak{f}$. On $C$ we have $K_C - \varphi^{\ast} K_{Y} \sim R_{\varphi}$, i.e. 
$\mathcal{O}_C (R_{\varphi}) = \omega_C \otimes (\varphi^{\ast} \omega_{Y})^{\vee}$. The canonical divisor $K_C$ on $C$ 
is induced by the restriction of $K_S + C \sim K_S + (2Y_0 + 2E \mathfrak{f})$ on $C$. Similarly, the restriction of $K_S + 
Y_1 \sim K_S 
+ Y_0 + E \mathfrak{f}$ on $Y_1$ induces $K_{Y_1}$. Therefore
\[
 R_{\varphi} \sim (K_S + (2Y_0 + 2E \mathfrak{f}) - (K_S + Y_0 + E \mathfrak{f}))_{|_C} \sim (Y_0 + E \mathfrak{f})_{|_C} \, .
\]
By {\rm (a)} and {\rm (c.1)},
$
  h^0 (C, \mathcal{O}_C ({Y_0 + E \mathfrak{f}}_{|_C})) = h^0 (S, \mathcal{O}_S (Y_0 + E \mathfrak{f})) = e - \gamma + 2 \, 
$.
Therefore the linear series $|\mathcal{O}_S (Y_0 + E \mathfrak{f})|$ on $S$ induces the linear series $|\mathcal{O}_C 
(R_{\varphi})|$ on 
$C$. 
\end{proof}

\begin{remark}
When $e = g - 2\gamma + 1 \geq 2\gamma - 1$ and the divisor $E$ on $Y$ is very ample, where $\mathcal{O}_Y (-E)$ is the 
Tschirnhausen module of a double covering $X \to Y$, statement {\rm (c.2)} implies that $\mathcal{O}_C (R_{\varphi})$ is very 
ample and $h^0 (C, \mathcal{O}_C (R_{\varphi})) = g - 3\gamma + 3$. It improves a similar claim proved in \cite[Lemma 
4.1]{CIK17} where it was assumed that $g \geq 6\gamma - 1$.
\end{remark}

\begin{remark}\label{Section_2_Remark_Fam_counting}
{\rm Proposition \ref{Prop_ruled_surface}} suggests how to give an alternative construction of the component $\mathcal{D}_{2g - 
4\gamma + 2, g, r}$ constructed in \cite[Theorem 4.3, p. 594]{CIK17}. For this take $e = g - 2\gamma + 1 \geq 2\gamma - 1$ and 
consider the family $\mathcal{Z}$ of surface scrolls $F \subset \mathbb{P}^R$, over a curve $Y$ of genus $\gamma$, $\deg F = 
\deg Y = e = g - 2\gamma + 1$ with $h^0 (F, \mathcal{O}_F (1)) = g - 3\gamma + 3$ and $h^1 (F, \mathcal{O}_F (1)) = 
\gamma$. According to \cite[Lemma 1, p. 7]{CCFM2008} such a scroll is necessarily a cone, say $F$, over a projectively normal 
curve in $\mathbb{P}^{R-1}$ of genus $\gamma$ and degree $e$. Further, let $\mathcal{F}$ be the family of smooth curves in 
$|\mathcal{O}_F (2)|$ on the cones $F \subset \mathbb{P}^{R}$ from the family $\mathcal{Z}$. By a counting of the parameters on 
which the family 
$\mathcal{Z}$ depends, similar to the one carried out in \cite[Remark 2, p. 15]{CCFM2008} and \cite[Proposition 7.1, p. 
150]{CCFM2009}, 
\begin{itemize}
 \item[$\phantom{.}$] $\dim \mathcal{Z} = $
    \begin{itemize}
     \item[$ + $] $3\gamma - 3$ \ : \ number of parameters of curves $Y \in \mathcal{M}_{ \gamma }$ 
     \item[$ + $] $\gamma$ \ : \ number of parameters of line bundles $\mathcal{O}_Y (E) \in \Pic (Y)$ of degree $g - 2\gamma + 
1 \geq 2\gamma - 1$ necessary to fix the geometrically ruled surface $\mathbb{P} (\mathcal{O}_Y \oplus \mathcal{O}_Y (-E))$
     \item[$ + $] $(R+1)^2 - 1 = \dim (\Aut (\mathbb{P}^R))$
     \item[$ - $] $((g - 2\gamma + 1) - \gamma + 2) = \dim G_F$, where $G_F$ is the subgroup of $\Aut 
(\mathbb{P}^R)$ fixing the scroll $F$, see \cite[Lemma 6.4, p. 148]{CCFM2009}
    \end{itemize}
\end{itemize}
one finds that 
$\dim \mathcal{Z} = 7(\gamma-1) - g + (R+1)^2$. On the other hand computing $\dim |\mathcal{O}_F 
(2)|$ using (\ref{S_decompose_cohomologies}) and the Riemann-Roch formula, we get easily $\dim |\mathcal{O}_F (2)| = 3g - 
8\gamma + 5$. Therefore for the dimension of $\mathcal{F}$ we obtain
\[
  \dim \mathcal{F} = \dim \mathcal{Z} + \dim |\mathcal{O}_F (2)| = 2g - \gamma - 2 + (g - 3\gamma + 3)^2.
\]
It is precisely the dimension of the component $\mathcal{D}_{2g-4\gamma+2, g, r}$ constructed in \cite[Theorem 4.3]{CIK17} 
when $r = R = g - 3\gamma + 2$ and it improves the bound calculated in \cite[Lemma 
4.1]{CIK17} where it was assumed that $g \geq 6\gamma - 1$.
\end{remark}

The above arguments do not imply yet that the family $\mathcal{F}$ gives rise to a component of the Hilbert scheme 
$\mathcal{I}_{d , g, R}$. To prove this formally, we will compute in section \ref{Section_4} $h^0 (C, N_{C / \mathbb{P}^R})$ for 
a general $C \in \mathcal{F}$. For the purposes of that computation we need several more formal statements about the normal 
bundles of curves on cones, which we prove below.

\begin{lemma}\label{Lemma_ConeProjNormalB}
Let $X$ be a smooth non-degenerate curve in $\mathbb{P}^{r}$ and let
$H$ be a hyperplane in $\mathbb{P}^{r}$. \, Assume that $\pi_p:X\to H \subset \mathbb{P}^{r}$ is a projection from a point $p\notin H\cup X$ such that the image $Y:=\pi_p(X)$ is smooth in $\mathbb{P}^{r-1}$. Then
\begin{equation}\label{ConeProjNormalB_SES}
  0 \to O_X(R_{\pi_p})\otimes\mathcal{O}_X (1) \to N_{X / \mathbb{P}^{r}} \to \pi_p^{\ast} N_{Y / \mathbb{P}^{r-1}} \to 0 \, ,
\end{equation}
where $R_{\pi_p}$ is the ramification divisor of the 
covering $\pi_p : X \to Y$.
\end{lemma}
\begin{proof}
Since $\pi_p:X\to \mathbb{P}^{r-1} \subset \mathbb{P}^{r}$ is a projection from a point $p\notin X$, we have $\pi_p^{\ast} 
(\mathcal O_Y(1))=\mathcal O_X(1)$. For the curves $X$ and $Y$ we have the Euler sequences
\[
     0 \to \mathcal{O}_X \to \oplus^{r+1} \mathcal{O}_X (1) \to T_{\mathbb{P}^r |_X} \to 0 \, 
\]
and 
\[
     0 \to \mathcal{O}_Y\to  \oplus^{r} \mathcal{O}_Y (1) \to T_{\mathbb{P}^{r-1} |_Y} \to 0 \, 
\]

\noindent Pulling the second sequence to $X$ via $\pi_p$ we obtain

{\small
\begin{equation*}\label{CommDiag_T_X_and_T_X}
\begin{array}{ccccccccccccccccccccccc}
  & & & & 0 & & 0 & & \\[1ex]
  & & & & \downarrow & & \downarrow & & \\[1ex]
  & &  & & \mathcal{O}_X(1) & \simeq  & \ker(\alpha) & &  \\[1ex]
  & & & & \downarrow & & \downarrow & \\[1ex]
  0 & \lra & \mathcal{O}_{X} & \rightarrow &  \oplus^{r+1}  \mathcal{O}_X (1) & \rightarrow & T_{\mathbb{P}^r |_X} &  
\rightarrow & 0 \\[1ex]
  & & \downarrow & & \downarrow & & \downarrow \alpha & \\[1ex]
  0 & \lra 	& \mathcal{O}_{X} & \rightarrow & \pi^{\ast}_p \left( \oplus^{r}_{1} \mathcal{O}_{Y}(1) \right) 
& \rightarrow & \pi^{\ast}_p \left( T_{\mathbb{P}^{r-1} |_{Y}}\right)   & \rightarrow & 0 \\[1ex]
  & &  & & \downarrow & &  & \\[1ex]
  & & & & 0 & &  & & \\[1ex]
\end{array}
\end{equation*}
}
\noindent where $\alpha$ is the induced map between the restrictions of $T_{\mathbb{P}^{r} |_X}$ and 
$\pi^{\ast}_p \left( T_{\mathbb{P}^{r-1} |_{Y}} \right)$ and $\ker \left( \alpha \right)$ is its kernel. By the Snake lemma we 
obtain
\[
 0 	\to \mathcal{O}_{X} (1) 
	\to T_{\mathbb{P}^r |_X}
	\to  \pi^{\ast}_p \left( T_{\mathbb{P}^{r-1} |_{Y}}\right)    
	\to 0 \, .
\]
Further, using the normal bundle sequence for $N_{X / \mathbb{P}^r}$ and $N_{{Y} / \mathbb{P}^{r-1}}$, 
we get the following commutative diagram
\begin{equation*}\label{CommDiag_N_X_and_N_X}
\begin{array}{ccccccccccccccccccccccc}
  & & & & 0 & & 0 & & \\[1ex]
  & & & & \downarrow & & \downarrow & & \\[1ex]
  & &  & & \mathcal{O}_X(1)  &  & \ker(\beta) & &  \\[1ex]
  & & & & \downarrow & & \downarrow & \\[1ex]
  0 & \lra & T_{X} & \rightarrow & T_{\mathbb{P}^r |_X} & \rightarrow & N_{X / \mathbb{P}^r} & \rightarrow & 0 \\[1ex]
  & & \downarrow & & \downarrow & & \downarrow \beta & \\[1ex]
  
  0  &	\to  & \pi^{\ast}_p(T_{Y}) 
 &	\to  & \pi^{\ast}_p \left( {T_{\mathbb{P}^{r-1}}}_{|_{Y}} \right) 
 &	\to  & \pi^{\ast}_p \left( N_{Y / \mathbb{P}^{r-1}} \right)
 &	\to  & 0
\\[1ex]
    & & \downarrow & & \downarrow & &  & \\[1ex]
  & &  \mathcal{O}_{R_{\pi_p}}  & & 0 & &  & & \\[1ex]
  & & \downarrow & & & & & \\[1ex]
  & & 0 & & & & & \\[1ex]
\end{array}
\end{equation*}
\noindent where $\beta$ is the induced map between the normal bundles $N_{X / \mathbb{P}^r}$ and $\pi^{\ast}_p \left( N_{Y / 
\mathbb{P}^{r-1}} \right)$. Similarly as before, by the Snake lemma we get $\ker \beta \cong \mathcal{O}_{X} 
(R_{\pi_p}) \otimes \mathcal{O}_X(1)$, and thus we 
deduce the short exact sequence \color{black}
\[
  0   \to \mathcal{O}_{X} (R_{\pi_p}) \otimes \mathcal{O}_X(1)
      \to N_{X / \mathbb{P}^r} \to \pi^{\ast}_{p} N_{Y / \mathbb{P}^{r-1}} \to 0
\]

\end{proof}

\begin{coro}\label{Corollary_ConeProjNormalB_1}
Suppose that $Y \subset \mathbb{P}^{r-1} \subset \mathbb{P}^{r}$, $r \geq 3$, is a smooth non-degenerate curve of genus 
$\gamma$. Let $p \in \mathbb{P}^{r} \setminus \mathbb{P}^{r-1}$ be and arbitrary point.  Consider the cone $F \subset 
\mathbb{P}^{r}$ over $Y$ with vertex $p$. Suppose that a curve $X \subset F$ is cut by a general hypersurface  
$Q_m \subset \mathbb{P}^{R}$ of degree $m$, i.e. $X \in 
|\mathcal{O}_F (m)|$ is general. Let $\varphi \, : \, X \xrightarrow{m:1} Y$ be the $m$-sheeted covering map induced by the 
ruling of the cone. Then there is an exact sequence 
\begin{equation}\label{ConeProjNormalB_SES_1}
  0 \to \mathcal{O}_X (m) \to N_{X / \mathbb{P}^{r}} \to \varphi^{\ast} N_{Y / \mathbb{P}^{r-1}} \to 0 \, .
\end{equation}
\end{coro}
\begin{proof}
The line bundle $O_X(R_\varphi)$ associated to the ramification divisor $R_{\varphi}$ of the covering $\varphi : X \to Y$ has 
the property $\mathcal O_X(R_\varphi) \simeq \mathcal O_X (m-1)$. To see this, recall that $R_{\varphi} \sim K_X - 
\varphi^{\ast} K_Y$. The canonical divisor $K_X$ on $X$ is cut by the restriction of $K_F + X$ on $X$ and $K_Y$ is cut by the 
restriction of $K_F + Y$ 
on $Y$. Therefore 
\[
  K_X - \varphi^{\ast}K_Y=(K_F + X)|_X - (K_F + Y)|_X \sim (X - Y)|_X \sim (m-1) Y|_X \, .
\]
Hence  $\mathcal O_X(R_\varphi)\simeq \mathcal O_X(m-1)$ and {\rm Lemma \ref{Lemma_ConeProjNormalB}} yields the exact sequence 
\eqref{ConeProjNormalB_SES_1}.
\end{proof}


\begin{coro}\label{Corollary_ConeProjNormalB_2}
Let $X, Y \subset F \subset \mathbb{P}^r$ be smooth curves on the cone $F$ with vertex $p$ as in {\rm Corollary 
\ref{Corollary_ConeProjNormalB_1}}, where $r \geq 6$. Let $W \subset \mathbb{P}^r$ be a general projective subspace of 
$\mathbb{P}^r$ of dimension $r - s - 1$, where \, $5 \leq s \leq r-1$. Consider the projection 
$
  \pi_{W} : \mathbb{P}^{r} \setminus W \to \mathbb{P}^{s}
$ 
with center $W$ to a general projective subspace of $\mathbb{P}^r$ of dimension $s$. Denote by $X_s$, $Y_s$ and $F_s$ the 
images of $X$, $Y$ and $F$ under $\pi_{W}$. Let 
$
  \varphi_s \, : \, X_s \to Y_s
$
be the covering map induced by the ruling of $F_s$. Then
\begin{equation}\label{ConeProjNormalB_SES_2}
  0 \to \mathcal{O}_{X_s} (m) \to N_{X_s / \mathbb{P}^{s}} \to \varphi^{\ast}_s N_{Y_s / \mathbb{P}^{s-1}} \to 0 \, .
\end{equation}
\end{coro}
\begin{proof}
Since $r \geq s+1 \geq 6$, a general projective subspace of $\mathbb{P}^r$ of dimension $r - s - 1$ does not meet the secant 
variety of $F$, which is of dimension at most $5$. Therefore $X$, $Y$ and $F$ are isomorphic to their images $X_s$, 
$Y_s$ and $F_s$. Also, the $m:1$ covering $\varphi : X \to Y$ induced by the ruling on $F$ goes to an $m:1$ 
covering $\varphi_s : X_s \to Y_s$ induced by the ruling on $F_s$ such that 
$
    {\pi_W}_{|_{Y}} \circ \varphi = \varphi_s \circ {\pi_W}_{|_{X}}
$.
In particular, ${\pi_W}_{|_{X}}(R_{\varphi}) = R_{\varphi_s}$. Thus the ramification divisor $R_{\varphi_s}$ is 
linearly equivalent to a divisor cut on $X_s$ by a hypersurface of degree $m-1$ in $\mathbb{P}^s$.
Hence $\mathcal O_{X_s} (R_{\varphi_s}) \simeq \mathcal O_{X_s} (m-1)$ and {\rm Lemma 
\ref{Lemma_ConeProjNormalB}} gives the exact sequence \eqref{ConeProjNormalB_SES_2}.
\end{proof}


\bigskip

\section{A short note on the Gaussian map}\label{Section_3}

Let $Y$ be a smooth curve of genus $\gamma$ and $L$ and $M$ be line bundles on $Y$. Let $\mu_{L,M}$
\begin{equation}\label{Sec3_mu_L,M}
 \mu_{L,M} \, : \, H^0 (Y, L) \otimes H^0 (Y, M) \to H^0 (Y, L \otimes M) 
\end{equation}
be the natural multiplication. The Gaussian map $\Phi_{L, M}$
\[
  \Phi_{L, M} \, : \, \ker \mu_{L,M} \to H^0 (Y, L \otimes M \otimes \omega_Y) 
\]
was introduced by Wahl in \cite{Wahl90}. Locally, $\Phi_{L, M} \, : \, s \otimes t \mapsto sdt - tds$ for sections $s \in H^0(L)$ 
and $t \in H^0(M)$. It has been studied by a number of authors. We refer to \cite{Wahl90} and \cite{CHM88} for its precise 
definition and some properties. We recall only several notions that will be used in {\rm Proposition \ref{Prob_Gauss_NBundle}} 
needed for the proof of {\rm Theorem A}.

The notation $R(L, M)$ is often used instead of $\ker \mu_{L, M}$ for the map $\mu_{L, M}$ in (\ref{Sec3_mu_L,M}). When $V 
\subset H^0 (Y, L)$ is a vector subspace and $M = \omega_Y$, the map $\mu_{L, M}$ in (\ref{Sec3_mu_L,M}) restricted on $V 
\otimes H^0 (Y, \omega_Y)$ will be denoted by $\mu_V$ and the Gaussian map restricted on $\ker \mu_V$ will be denoted by 
$\Phi_{\omega_Y, V}$. 

The proposition that follows is formulated in the specific form in which it will be used in the proof of {\rm Theorem A}.

\begin{prop}\label{Prob_Gauss_NBundle}
Let $Y$ be a smooth curve of general moduli of genus $\gamma \ge 10$, and let $E$ be a general line bundle on $Y$ of degree 
$ g - 2\gamma +1 \geq 2\gamma-1$. Let $V \subseteq H^0 (Y, E)$ be general linear subspace of dimension 
$
  r = \dim V \geq \max \left\lbrace \gamma, \frac{2(g-1)}{\gamma} \right\rbrace
$. 
Consider the embedding $Y \subset \mathbb{P}^{r-1} \equiv \mathbb{P} (V^{\vee})$ given by $V$. Then 
\begin{itemize}
 \item the restricted Gaussian mapping $\Phi_{\omega_Y, V}$ is surjective, and
 \item $h^0 (N_{Y/\mathbb{P}^{r-1}}(-1)) = \dim V = r$.
\end{itemize}
\end{prop}
\begin{proof}
Denote by $\mu$ the cup-product map
\[
  \mu: H^0(Y, E) \otimes H^0(Y, \omega_Y) \to H^0(Y, \omega_Y\otimes E) \, .
\]
Since $\deg E = g-2\gamma+1 \geq 2\gamma - 1$, so $E$ is very ample, it follows by \cite[Theorem (4.e.1) and Theorem 
(4.e.4)]{Green84} and \cite{Cili83} that $\mu$ is surjective.

The linear series determined by $V$ is very ample since $Y \in \mathcal{M}_{\gamma}$ is general, 
$\gamma \ge 10$ and $V \subset H^0 (Y, E)$ is also general. Consider the restriction $\mu_V$ of $\mu$ to 
\[
  \mu_V : V \otimes H^0(Y, \omega_Y) \to H^0(Y, \omega_Y\otimes E) \, .
\]
Let $R(\omega_Y, E)$ be the kernel of the map $\mu$ and consider the Gaussian map 
$\Phi_{\omega_Y, E}$ defined on $R(\omega_Y, E)$
\[
  \Phi_{\omega_Y, E} : R (\omega_Y, E) \to H^0 (\omega_Y^2 \otimes E) \, ,
\]
and similarly its restriction $\Phi_{\omega_Y, V}$ defined on the kernel $R(\omega_Y, V)$ of the map $\mu_V$
\begin{equation}\label{Gauss_map_restrict}
  \Phi_{\omega_Y, V} : R (\omega_Y, V) \to H^0 (\omega_Y^2 \otimes E) \, .
\end{equation}
In the case of complete embedding, i.e. if $V = H^0 (Y, E)$, the claim follows by 
\cite[Proposition 1.2]{CilMir1990}, where it is proven that 
\[
    h^0 (N_{Y/\mathbb{P}^{r-1}}(-1)) = h^0 (Y, E) + \corank \left( \Phi_{\omega_Y, E} \right) \, ,
\]
and  by \cite[Proposition (2.9)]{CLM96}, where it is proven that $\Phi_{\omega_Y, E}$ is surjective for $\gamma\geq 10$ and 
$\deg E = g - 2\gamma + 1 \geq 2\gamma -1$. In the case of incomplete embedding, i.e. if $V \subsetneq H^0 (Y, 
E)$, exactly the same argument as in the proof of \cite[Proposition 1.2]{CilMir1990} shows that
\begin{equation}\label{Coho_dim_N(-1)_incompele}
    h^0 (N_{Y/\mathbb{P}^{r-1}}(-1)) 
	  = \dim V + \corank \left( \Phi_{\omega_Y, V} \right) 
	  = r + \corank \left( \Phi_{\omega_Y, V} \right) \, ,
\end{equation}
provided that $\mu_V$ is surjective. This is what we will prove next.

Since $\mu_V$ is the restriction of $\mu$ to $V \otimes H^0(Y, \omega_Y)$, we have
\[
 \ker \mu_V = \ker \mu \cap \left(V \otimes H^0(Y, \omega_Y)\right) \, .
\]
Due to $\gamma \leq \dim V \leq \dim H^0 (Y, E)$, it follows from \cite[Proposition 4.3]{Bal1995} that
\begin{equation}\label{dim_ker_muV}
 \dim \left(\ker \mu \cap \left(V \otimes H^0(Y, \omega_Y)\right) \right) 
      = \max \{0, \dim \left( \ker \mu \right) - (h^0 (Y,E) - \dim V) h^0 (Y, \omega_Y) \} \, .
\end{equation}
Since $\mu$ is surjective, 
$
    \dim \left( \ker \mu \right)
	  = (\deg(E)-\gamma+1)\gamma - (\deg(E) + \gamma - 1)
	  = (g-3\gamma)(\gamma-1)
$. By assumption $r = \dim V \geq \frac{2(g-1)}{\gamma}$, hence 
\[
\begin{aligned}
 \dim \ker \mu - (h^0 (Y, E) - \dim V)\gamma & = (\gamma - 1)(g - 3\gamma) - (g - 3\gamma + 2 - r)\gamma \\
    & = \gamma - g + r\gamma > 0 \, .
\end{aligned}
\]
By (\ref{dim_ker_muV}) we obtain
\[
 \dim \ker \mu_V = \gamma - g + r\gamma \, .
\]
From here we get for the dimension of its image
\[
 \dim \left( \Im (\mu_V) \right) = r\gamma - \dim \ker \mu_V = g - \gamma = h^0 (Y, \omega_Y \otimes E) \, .
\]
This shows that $\mu_V$ is surjective, which proves (\ref{Coho_dim_N(-1)_incompele}).

It remains to show that $\Phi_{\omega_Y, V}$ is surjective. According to \cite[Theorem 4.1]{Bal1995}, the Gaussian map 
$\Phi_{\omega_Y, V}$ is of maximal rank. Suppose that it is not surjective. Then it must be injective and its image in $H^0 
(Y, \omega_Y^2 \otimes E)$ should be proper, hence
\[
 \gamma - g + r\gamma = \dim \ker \mu_V < h^0 (Y, \omega_Y^2 \otimes E) = g + \gamma - 2 \, ,
\]
which implies $r < \frac{2(g-1)}{\gamma}$. The last is impossible in view of the assumption that $r = \dim V \geq \max 
\left\lbrace \gamma, \frac{2(g-1)}{\gamma} \right\rbrace$. Therefore, $\Phi_{\omega_Y, V}$ must be surjective and from 
(\ref{Coho_dim_N(-1)_incompele}) we conclude also that $h^0 (N_{Y/\mathbb{P}^{r-1}}(-1)) = \dim V = r$.
\end{proof}

\bigskip

\section{Proof of Theorem A}\label{Section_4}

Before demonstrating the proof of {\rm Theorem A} we recall a few facts concerning the Hilbert scheme of cones.
{\rm Proposition \ref{Prop_ruled_surface}} and the counting of the number of parameters in Remark 
\ref{Section_2_Remark_Fam_counting} gives the idea how to construct explicitly the component $\mathcal{D}_{d, g, R}$. Recall 
that $d = 2g - 4\gamma + 2$ and $R = g - 3\gamma + 2$.

Let $\gamma \geq 10$ and $g \geq 4\gamma - 2$ be integers. Consider the Hilbert scheme $\mathcal{I}_{d/2, \gamma, R-1}$ of 
smooth curves of degree $d/2$ and genus $\gamma$ in $\mathbb{P}^{R - 1}$. By \cite[Theorem on p. 75]{Har82} and \cite[Theorem on 
p. 26]{Ser84}, \ $\mathcal{I}_{d/2, \gamma, R - 1}$ is reduced and irreducible of dimension 
$
  \lambda_{d/2, \gamma, R - 1} = Rd/2 - (R-4)(\gamma - 1)
$. Denote by 
$\mathcal{H}(\mathcal{I}_{d/2, \gamma, R-1})$ the family of cones in $\mathbb{P}^{R}$ over curves representing points of 
$\mathcal{I}_{d/2, \gamma, R-1}$. Since $\gamma \geq 10$ it follows by \cite[Proposition 2.1]{CLM96} that for a general $[Y] 
\in \mathcal{I}_{d/2, \gamma, R-1}$ the Gaussian map $\Phi_{\omega_Y, \mathcal{O}_Y (1)}$ is surjective, 
hence by \cite[Proposition 2.18]{CLM96} $\mathcal{H}(\mathcal{I}_{d/2, \gamma, R-1})$ is a generically 
smooth component of the Hilbert scheme of surfaces of degree $d/2$ in $\mathbb{P}^{R}$ and 
\begin{equation}\label{Sec4_dim_scheme_of_cones}
  \dim \mathcal{H}(\mathcal{I}_{d/2, \gamma, R-1}) 
      = h^0 (Y, N_{Y / \mathbb{P}^{R - 1}}) + R
      = \lambda_{d/2, \gamma, R-1} + R \, .
\end{equation}

First we give the proof of {\rm Theorem A} in the case $r = R$.

\begin{prop}\label{Prop_construct_DdgR}
Suppose that $\gamma \geq 10$ and $g \geq 4\gamma - 2$. Let $\mathcal{F}_{d, g, R}$ be the family of curves $C \subset 
\mathbb{P}^R$ obtained as the intersection of a cone $F$ and a general hypersurface of degree 2 in $\mathbb{P}^R$, where $[F] 
\in \mathcal{H}(\mathcal{I}_{d/2, \gamma, R-1})$. Let $\mathcal{D}_{d, g, R}$ be the closure of the set of points in 
$\mathcal{I}_{d, g, R}$ corresponding to curves from the family $\mathcal{F}_{d, g, R}$. Then 
\begin{itemize}
 \item $\mathcal{D}_{d, g, R}$ is a generically smooth irreducible component of $\mathcal{I}_{d, g, R}$, and 
 \item $\dim \mathcal{D}_{d, g, R} = 2g - \gamma - 2 + (R+1)^2 = \lambda_{d, g, R} + R\gamma - 2g + 2$.
\end{itemize}
\end{prop}
\begin{proof}
First we compute $\dim \mathcal{D}_{d, g, R}$. For a general point $[F] \in \mathcal{H}(\mathcal{I}_{d/2, \gamma, R - 1})$, the 
cone $F$ is projectively normal since it is a cone over a general curve $Y$ from $\mathcal{I}_{d/2, \gamma, R-1}$, 
which is projectively normal by \cite[Theorem 1, p. 74]{GL1986}. Therefore the linear series $|\mathcal{O}_F (2)|$ on $F$ is 
induced by $|\mathcal{O}_{\mathbb{P}^R} (2)|$. By equalities (\ref{Sec4_dim_scheme_of_cones}) and 
(\ref{S_decompose_cohomologies}), $h^0 (F, \mathcal{O}_F (2)) = 3g - 8\gamma + 6$. Therefore

\[
  \begin{aligned} 
    \dim \mathcal{D}_{d, g, R} & = \dim \mathcal{H}(\mathcal{I}_{g - 2\gamma + 1, \gamma, R - 1}) + h^0 (F, \mathcal{O}_F (2)) 
- 1 \\
      & = \lambda_{d/2, \gamma, R - 1} + R + 3g - 8\gamma + 5 \, .
  \end{aligned}
\]
\noindent Remark that since $\lambda_{d/2, \gamma, R - 1} = Rd/2 - (R-4)(\gamma - 1)$ and $\lambda_{d, g, R} = (R+1)d - 
(R-1)(g-1)$, the expression for $\dim \mathcal{D}_{d, g, R}$ can also be written as

\[
  \dim \mathcal{D}_{d, g, R} = \lambda_{d, g, R} + R\gamma - 2g + 2 = (R+1)^2 + 2g - \gamma - 2 \, .
\]

To prove that $\mathcal{D}_{d, g, R}$ is a generically smooth component of $\mathcal{I}_{d, g, R}$, it is sufficient to show that 
for a general $[X] \in \mathcal{D}_{d, g, R}$ we have 
$
  h^0 (X, N_{X/\mathbb{P}^r}) = (R+1)^2 + 2g - \gamma - 2 = \dim \mathcal{D}_{d, g, R}.
$
Since $X \subset F$ is cut by a general quadratic hypersurface in $\mathbb{P}^R$, the ruling of $F$ 
induces a double covering $\varphi : X \to Y$, where $Y \subset F$ is cut by a general hyperplane in $\mathbb{P}^R$ and also 
$[Y] \in \mathcal{I}_{d/2, \gamma, R-1}$ is general. It follows by {\rm Proposition \ref{Prop_ruled_surface}} and {\rm 
Corollary \ref{Corollary_ConeProjNormalB_1}} that 
\[
  0 \to \mathcal{O}_X (2) \to N_{X / \mathbb{P}^{R}} \to \varphi^{\ast} N_{Y / \mathbb{P}^{R-1}} \to 0 \, .
\] 
Since 
$
  \deg \mathcal{O}_X (2) = 2d = 4g - 8\gamma + 4 > 2g - 2
$
, the series $|\mathcal{O}_X (2)|$ is nonspecial, hence $h^1 (X, \mathcal{O}_X (2)) = 0$. Therefore, using projection formula, Leray's isomorphism and $\varphi_*O_X = O_Y+O_Y(-E)$, we get 
\[
 \begin{aligned}
  h^0 (X, N_{X / \mathbb{P}^{R}}) 
      & = h^0 (X, \mathcal{O}_X (2)) + h^0 (X, \varphi^{\ast} N_{Y / \mathbb{P}^{R-1}}) \\
      & = h^0 (X, \mathcal{O}_X (2)) + h^0 (Y, N_{Y / \mathbb{P}^{R-1}}) + h^0 (Y, N_{Y / \mathbb{P}^{R-1}}(-1)) \\
      & =  3g - 8\gamma + 5 + \lambda_{d/2, \gamma, R-1} + R  \\
      & =  \dim \mathcal{D}_{d, g, R} \, .
 \end{aligned}
\] 
This implies that for a general $[X] \in \mathcal{D}_{d,g,R}$ 
\[
  \dim \mathcal{D}_{d,g,R} = \dim \mathcal{F}_{d,g,R} = h^0 (X, N_{X / \mathbb{P}^{R}}) = \dim T_{[X]} \mathcal{D}_{d,g,R} \, ,
\]
therefore $\mathcal{D}_{d,g,R}$ is a generically smooth component of $\mathcal{I}_{d,g,R}$.
\end{proof}

Now we give the proof of {\rm Theorem A} for $\max \left\lbrace \gamma, \frac{2(g-1)}{\gamma} \right\rbrace \leq r < R$.

\begin{proof}[{\bf Proof of { Theorem A}}]
Let $\mathcal{F}_{d, g, r}$ be the family of curves in $\mathbb{P}^r$ obtained from the family 
$\mathcal{F}_{d, g, R}$ in {\rm Proposition \ref{Prop_construct_DdgR}} by a projection $\pi_W : \mathbb{P}^R \to \mathbb{P}^r$ 
with center $W \subset \mathbb{P}^R$, where $W \cong \mathbb{P}^{R-r-1}$ is general. 
Let $\mathcal{D}_{d, g, r}$ be the closure of the set of points in $\mathcal{I}_{d, g, r}$ corresponding to the curves from the 
family $\mathcal{F}_{d, g, r}$. 
Since $\codim (W, \mathbb{P}^R) \geq \gamma+1 \geq 10$, a cone $F$ and curves $X$ and $Y$ as in the proof of {\rm Proposition 
\ref{Prop_construct_DdgR}} are isomorphic to their images $F_r = \pi_W (F)$, $X_r = \pi_W (X)$ and $Y_r = \pi_W (Y)$, 
correspondingly. 
Also, $\varphi_r : X_r \to Y_r$ induced by ruling of $F_r$ is a double covering as 
in {\rm Corollary \ref{Corollary_ConeProjNormalB_2}}. Note that if $p$ is the vertex of $F$ then $\pi_W (p)$ is the vertex of 
$F_r$. Therefore 
\begin{equation}\label{TheoremA_proof_SES_Normal_bundles}
    0 \to \mathcal{O}_{X_r} (2) \to N_{X_r / \mathbb{P}^{r}} \to {\varphi^{\ast}_r} N_{Y_r / \mathbb{P}^{r-1}} \to 0 \, .
\end{equation} 
The embedding of $Y_r \subset \mathbb{P}^{r-1}$ is incomplete, but since $\deg Y_r = d/2 = g - 2\gamma + 1  \geq 2\gamma - 1$, 
it follows by \cite{Ser84} that $\mathcal{I}_{d/2, \gamma, r-1}$ has a unique generically smooth component of the expected 
dimension $\lambda_{d/2, \gamma, r-1}$. Therefore, for a general $Y_r \in \mathcal{I}_{d/2, \gamma, r-1}$ (as in our case), $h^1 
(Y_r, N_{Y_r / \mathbb{P}^{r-1}}) = 0$ or equivalently $h^0 (Y_r, N_{Y_r / \mathbb{P}^{r-1}}) = \lambda_{d/2, \gamma, r-1}$. 
Then we can compute $h^0 (N_{X_r / \mathbb{P}^r})$ in a similar way as before. Since the projection $\pi_W$ is general and 
$
    r \geq \max \left\lbrace \gamma, \frac{2(g-1)}{\gamma} \right\rbrace
$,
it follows by {\rm Proposition \ref{Prob_Gauss_NBundle}} that $h^0 (Y_r, N_{Y_r / \mathbb{P}^{r-1}}(-1)) = r$. 
Since $\deg \mathcal{O}_{X_r} (2) = 2d > 2g - 2$ we have 
$
  h^1 (X_r, \mathcal{O}_{X_r} (2)) = h^1 (X, \mathcal{O}_{X} (2)) = 0
$
. Using (\ref{TheoremA_proof_SES_Normal_bundles}) we find 
\[
 \begin{aligned}
  h^0 (X_r, N_{X_r / \mathbb{P}^{r}}) 
      & = h^0 (X_r, \mathcal{O}_{X_r} (2)) + h^0 (X_r, {\varphi^{\ast}_r} N_{Y_r / 
\mathbb{P}^{r-1}}) \\
      & =  2d - g + 1 + h^0 (Y_r, N_{Y_r / \mathbb{P}^{r-1}}) + h^0 (Y_r, N_{Y_r / 
\mathbb{P}^{r-1}} (-1)) \\
      & =  4g - 8\gamma + 4 - g + 1 + \lambda_{g - 2\gamma + 1, \gamma, r - 1} + r \\
      & =  3g - 8\gamma + 5 + r + (g - 2\gamma + 1)r - (r-4)(\gamma - 1) \\
      & =  (r+3)g - (3r+4)\gamma + 3r+1 \, . 
 \end{aligned}
\] 
Let's compute also the dimension of the family $\mathcal{F}_{d,g,r}$. It is similar to the one carried out in the proof of 
\cite[Theorem 4.3]{CIK17}. Since the curves in $\mathcal{F}_{d,g,r}$ are obtained as generic projections from $\mathbb{P}^R$ to 
$\mathbb{P}^r$, we have
\[
\begin{aligned}
  \dim \mathcal{F}_{d,g,r} 
      & = \dim \mathcal{F}_{d,g,R} - \dim\Aut \mathbb{P}^R + \dim\Aut \mathbb{P}^r + \dim Grass (r+1, R+1) \\
      & = 2g - \gamma - 2 + (r+1)^2 + (r+1)(R-r) \\
      & = 2g - \gamma - 2 + (r+1)(g - 3\gamma + 2 + 1) \\
      & = (r+3)g - (3r+4)\gamma + 3r + 1 \, .
\end{aligned}
\]
Notice that this number is exactly equal to the one claimed in {\rm Theorem A} since
\[
\begin{aligned}
  \lambda_{d, g, r} + r\gamma - 2g+2 
      & = (r+1)(2g - 4\gamma + 2) - (r-3)(g - 1) + r\gamma - 2g + 2 \\
      & = (r+3)g - (3r+4)\gamma + 3r + 1 \, .
\end{aligned}
\]
Hence $\dim \mathcal{D}_{d,g,r} = \dim \mathcal{F}_{d,g,r} = h^0 (X_r, N_{X_r / \mathbb{P}^{r}}) = \dim 
T_{[X_r]} \mathcal{D}_{d,g,r}$, which completes the proof of {\rm Theorem A}.
\end{proof}

\begin{coro}\label{Coro_Reg_comp_exists}
If
\begin{equation}\label{condition}
  \gamma ~|~ 2(g-1) \quad \mbox{ and } \quad r := \frac{2(g-1)}{\gamma} \geq \gamma \geq 10,
\end{equation} 
then $\mathcal{D}_{d, g, r}$ is a regular component of $\mathcal{I}_{d, g, r}$ different from the distinguished one.
\end{coro}
\begin{proof}
With the particular values of $d = 2g - 4\gamma + 2$ and $r$ we have $r = \frac{2(g-1)}{\gamma} \leq \frac{g-1}{5}$ for $\gamma 
\geq 10$, hence $5r \leq g-1$. From here it is easy to see that $d - g - r \geq g+2 - 5r \geq 3$. 
Therefore $\rho(d, g, r) = g - (r+1)(g-d+r) \geq g > 0$, hence the distinguished component of $\mathcal{I}_{d, g, r}$ 
dominating $\mathcal{M}_g$ exists. Apart 
from it, {\rm Theorem A} guarantees the existence of the regular component $\mathcal{D}_{d, g, r}$ which is apparently 
different from the distinguished one as the former projects properly in $\mathcal{M}_g$.
\end{proof}

\begin{remark}
It appears that the condition $r \geq \frac{2(g-1)}{\gamma}$ is essential for the family $\mathcal{F}_{d, g, r}$ giving rise to 
a component of $\mathcal{I}_{d, g, r}$, because in such case $r < \frac{2(g-1)}{\gamma}$ we have $\dim \mathcal{F}_{d, g, r} < 
\lambda_{d,g,r}$. Notice also that in such a case the Gaussian map in {\rm Proposition \ref{Prob_Gauss_NBundle}} is definitely 
not surjective.
\end{remark}

\begin{remark}\label{Sec3_2_LinNorm_comp_exists}
In their paper \cite{MS1989} Mezzetti and Sacchiero constructed generically smooth irreducible components of $\mathcal{I}_{d, g, 
r}$, denoted there $W^{m}_{d,g,r}$, whose general points are $m-$ sheeted coverings of $\mathbb{P}^1$, where $m \geq 3$. 
In the case of $g = 6\gamma - 3$, we have $d = 2g - 4\gamma + 2 = 8\gamma - 4 = \frac{4}{3}g$, $R = g - 3\gamma + 2 = 3\gamma - 
1 = \frac{g+1}{2}$, and the Hilbert scheme $\mathcal{I}_{\frac{4g}{3}, g, \frac{g+1}{2}}$ has two components parametrizing 
linearly normal curves. One of them is the component $\mathcal{D}_{\frac{4g}{3}, g, \frac{g+1}{2}}$, shown to exist in our
Proposition \ref{Prop_construct_DdgR}, and the other one is the component $W^{m}_{d,g,r}$ for $m = 4$, $d = \frac{4}{3}g$, 
and $r = R = \frac{g+1}{2}$ (it is easy to check that the numerical conditions for the existence of $W^{4}_{\frac{4g}{3}, g, 
\frac{g+1}{2}}$ are satisfied when $\gamma \geq 10$). Notice that since $d - g - R = - \gamma < 0$, the existence 
of these two components do not provide a counterexample to \emph{Severi's conjecture} claiming that $\mathcal{I}_{d, g, r}$ has 
a unique irreducible component parametrizing linearly normal curves if $d \geq g+r$.
\end{remark}

\vfill



\end{document}